\documentclass[10 pt]{amsart}
\usepackage{graphicx,amsmath,mathtools}
\usepackage{amssymb,epsfig,color,lastpage}
\usepackage{epstopdf,wrapfig,marvosym}
\usepackage[allcolors = blue,colorlinks]{hyperref}
\usepackage[utf8]{inputenc} 
\usepackage[top=1in, left=1in, right=1in, bottom=1in]{geometry}
\usepackage{libertine}
\usepackage[libertine]{newtxmath}
\usepackage{amsthm,amsaddr,sidecap,extarrows}
\usepackage{fancyhdr}
\hypersetup{pdfstartview={XYZ null null 1.00}}

\newtheorem{theorem}{Theorem}

\newtheorem{conjecture}[theorem]{Conjecture}

\newtheorem{question}[theorem]{Question}
\newtheorem{remark}[theorem]{Remark}

\begin{document}
	
\title{On odd torsion in even Khovanov homology}
\author{Sujoy Mukherjee}
\address{Department of Mathematics, George Washington University, Washington DC, USA}
\address{sujoymukherjee.math@gmail.com}


\date{January 14, 2018 and, in revised form, \today.}

	
\begin{abstract}
	
	This short note resolves the most important part of the PS braid conjecture while introducing the first known examples of knots and links with odd torsion of order 9, 27, 81, and 25 in their even Khovanov homology.
	
\end{abstract}

\maketitle

\hspace*{25pt}In 1999, the Jones polynomial \cite{Jon} was categorified into Khovanov homology (KH) \cite{Kho}. A bigraded chain complex is associated to a link whose homology is an invariant of the link itself. Additionally, the Euler characterictic of this chain complex, interpreted appropriately, is the Jones polynomial. KH is a powerful link invariant. In particular, it detects the unknot, which, for the Jones polynomial is still an open question \cite{KM}. KH has been used to prove the Milnor unknotting conjecture combinatorially among other applications \cite{Ras}. 

\

\hspace*{25pt}A wealth of information, in the form of torsion subgroups, is obtained from the KH of a link which has no contribution towards the Euler characteristic of the chain complex and hence cannot be obtained from just the Jones polynomial in general. It is thought that KH inherits $\mathbb{Z}_2$ torsion from its nature of construction and hence it is not surprising that almost all knots and links have this group in their KH \cite{AP,Shu}.

\

\hspace*{25pt}Things become more interesting and significantly harder when one tries to search for torsion subgroups of higher order. Not only are these extremely rare to find, computational limitations create roadblocks frequently. Until the article \cite{MPSWY}, only a finite examples of knots and prime links were known with non-$\mathbb{Z}_2$ torsion in their KH. In the aforementioned article, infinite families of knots and prime links containing $\mathbb{Z}_3,$ $\mathbb{Z}_4,$ $\mathbb{Z}_5,$ $\mathbb{Z}_7,$ and $\mathbb{Z}_8$ torsion in their KH were introduced. These families consisted of torus links with twists. In addition to this, links containing large torsion subgroups in their KH were also introduced. The orders of these subgroups are all powers of two. This led to the present quest of trying to find knots and links with large torsion subgroups of odd order in their KH. The goal of this article is to trigger a study of odd torsion in even KH while keeping a very important question in mind.

\

\begin{center}What does torsion in KH signify?\end{center}

\

\hspace*{25pt} For completeness, a very brief description of KH is first provided below. For more details, the reader is directed to the articles \cite{BN,Tur,Vir}. This is followed by the discussion on the PS braid conjecture. The final part of the paper introduces knots and links with odd torsion of order 9, 27, 81, and 25 in their even KH.

\begin{figure}[ht]
	
	\centering
	\includegraphics[width = 0.8\textwidth]{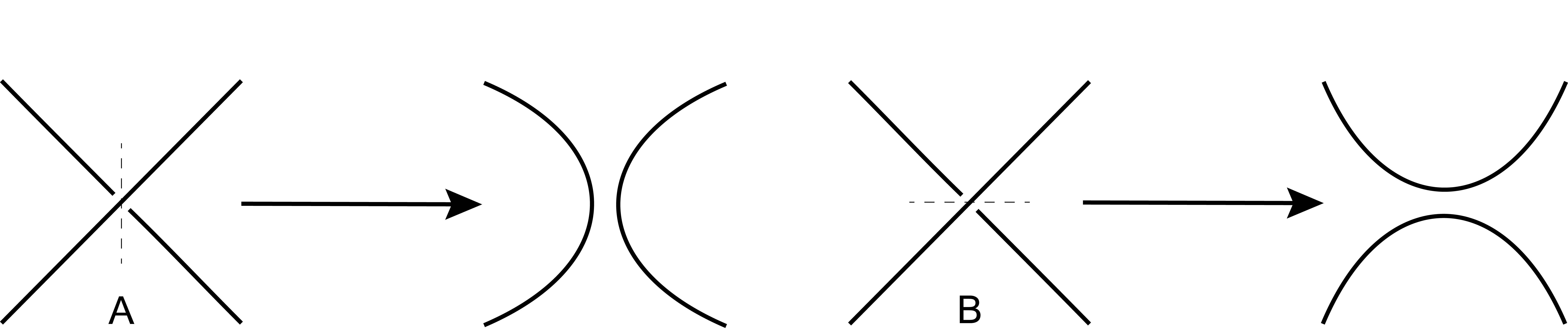}
	\caption{Smoothing conventions for Kauffman states.}
	\label{smoothing_convention}
	
\end{figure}

\section{Introduction}

\hspace*{25pt}For $D$, an unoriented diagram of a link, let $cro(D)$ denote its set of crossings. $s: cro(D) \longrightarrow \{A,B\}$ is called a Kauffman state of $D$. Let $\mathcal{KS}(D)$ denote the set of all Kauffman states. Further, let $\sigma = \sigma(s) = |s^{-1}(A)| - |s^{-1}(B)|$ for all $s \in \mathcal{KS}(D)$.

\
	
\hspace*{25pt}Denote by $cir(s)$, the set of circles obtained from a Kauffman state $s$ by smoothing the crossings of $D$ according to the convention described in Figure \ref{smoothing_convention} and let $t: cir(s) \longrightarrow \{-1,1\}.$ $t$ is called an enhanced Kauffman state of $D$. Denote the set of enhanced Kauffman states of a link diagram $D$ by $\mathcal{EKS}(D).$ For each enhanced Kauffman state $t,$ let $\tau = \tau(t) = \sum\limits_{c \in cir(s)} t(c).$

\
	
\hspace*{25pt}For an enhanced Kauffman state $t$ obtained from a Kauffman state $s,$ let $(a(t),b(t)) = (\sigma(s),\sigma(s) + 2\tau(t)).$ Now, let $C_{a,b}(D)$ be the free Abelian group with generators: $\{t \mid t \in \mathcal{EKS(D)}$ with $a(s) = a$ and $b(t) = b\}.$

\
	
\hspace*{25pt}Let $t$ and $t'$ be enhanced Kauffman states. Then, $t'$ is said to be adjacent to $t,$ if $a(t) = a(t') + 2,$ $b(t) = b(t')$ and the Kauffman states corresponding to $t$ and $t'$ differ at exactly one crossing with $t$ having an $A$ value and $t'$ having a $B$ value assigned to the crossing.

\
	
\hspace*{25pt}Finally, order the crossings of $D$ and for a fixed integer $b,$ consider the descendant complex:
		$$
		\cdots \longrightarrow C_{a+2,b}(D) \stackrel{\partial_{a+2,b}}{\longrightarrow} C_{a,b}(D) \stackrel{\partial_{a,b}}{\longrightarrow} C_{a-2,b}(D) \longrightarrow \cdots
		$$
		with differential $\partial_{a,b}(t) = \sum\limits_{t' \in \mathcal{EKS}} (t:t') t'.$ If $t'$ is not adjacent to $t$, $(t:t') = 0$ and otherwise $(t:t') = (-1)^\omega$, with $\omega$ the number of $B$-labeled crossings in $t$ coming after the crossing at which the Kauffman states corresponding to $t$ and $t'$ differ.
		
\

\hspace*{25pt}With the notation above, the groups $ H_{a,b}(D)=\frac{\textnormal{ker} (\partial_{a,b})}{\textnormal{im}(\partial_{a+2,b})}$ are invariant under Reidemeister moves II and III. These groups categorify the unreduced Kauffman bracket polynomial \cite{Kau} and are called the framed Khovanov homology groups.

\
	
\hspace*{25pt}Let $\vec D$ be the diagram obtained after fixing an orientation for $D$ and let $w = w(\vec D)$ be its writhe. Then, the classical Khovanov (co)homology of $\vec D$ can be obtained from the framed Khovanov homology in the following way: $$H^{i,j}(\vec D) = H_{w-2i,3w-2j}(D).$$

\hspace*{25pt}To avoid cumbersome notation, in the remaining part of this article, $D$ (instead of $\vec D$) is used to denote an oriented link diagram.

\begin{figure}[ht]
	
	\centering
	\includegraphics[width = 0.3\textwidth]{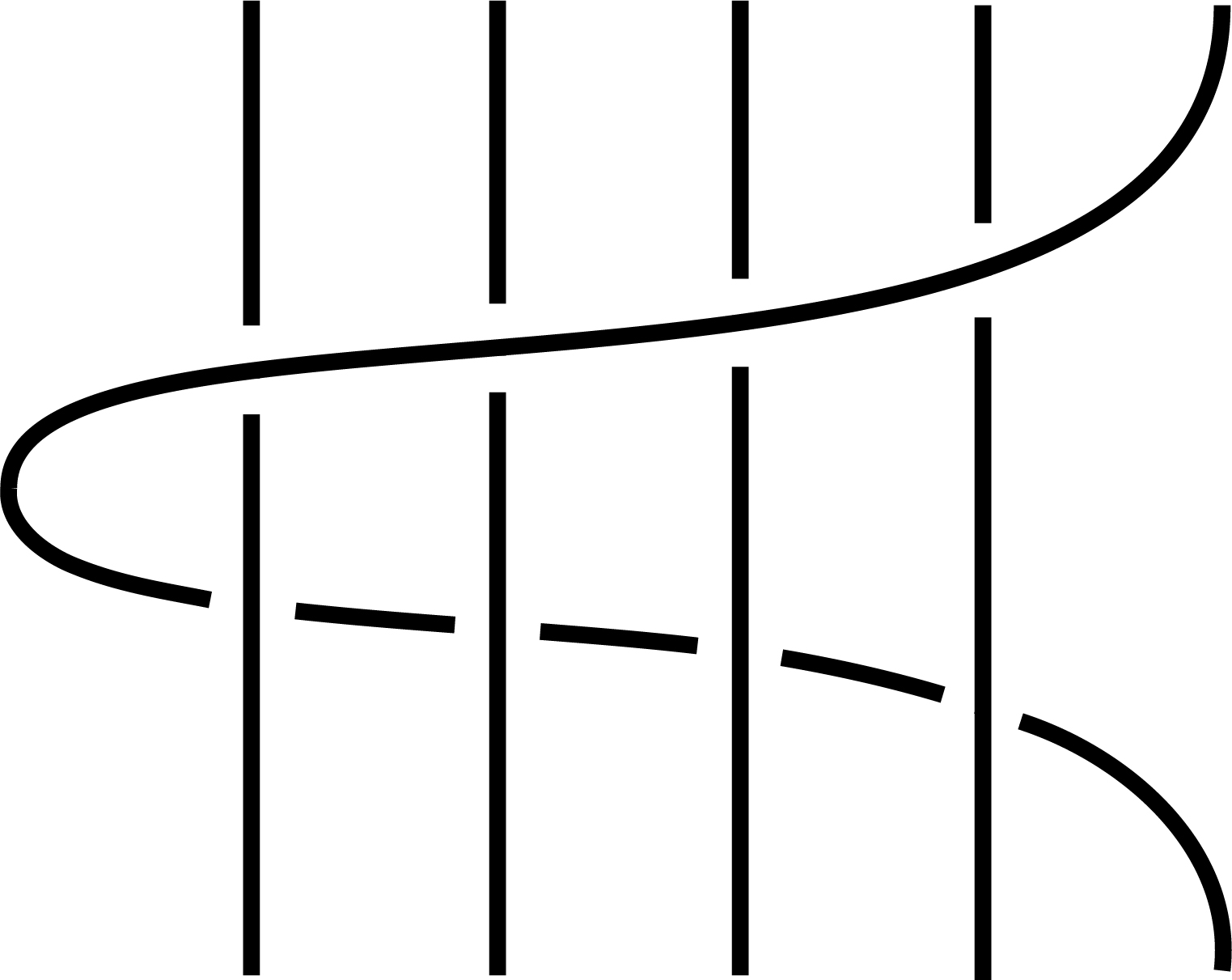}
	\caption{The element $w_{1,4} = \sigma_1 \sigma_2 \sigma_3 \sigma_4 \sigma_4 \sigma_3 \sigma_2 \sigma_1 \in \mathcal{B}_5.$}
	\label{wrap_braid}
	
\end{figure}

\hspace*{25pt}Most of the previously known examples of knots and links with interesting torsion are derived from torus links. In fact, it is conjectured that torus knots of type $(p,p+1),$ where $p$ is a prime, contain $\mathbb{Z}_p$ torsion. The trend of relying on derivatives of torus links for torsion in KH continues in this note too, however, with modifications to torus links in a different way. In fact, the author hopes that the KH data of the links in these families will encourage the readers to compute their other invariants.

\

\hspace*{25pt}Let $\mathcal{B}_n$ denote the braid group with $(n-1)$ generators. Further, let $0 \leq i \neq j < n$. For $i<j,$ denote by $w_{i,j},$ the braid word: $\sigma_i \sigma_{i+1} \cdots \sigma_j \sigma_j \sigma_{j-1} \cdots \sigma_i$ and similarly, by $w_{i,j}^{-1},$ the braid word: $\sigma_i^{-1} \sigma_{i+1}^{-1} \cdots \sigma_j^{-1} \sigma_j^{-1} \sigma_{j-1}^{-1} \cdots \sigma_i^{-1}.$ Figure \ref{wrap_braid} shows $w_{1,4} \in \mathcal{B}_5.$ Analogously, for $i>j,$ denote by $w_{i,j},$ the braid word: $\sigma_i \sigma_{i-1} \cdots \sigma_j \sigma_j \sigma_{j+1} \cdots \sigma_i$ and by $w_{i,j}^{-1},$ the braid word: $\sigma_i^{-1} \sigma_{i-1}^{-1} \cdots \sigma_j^{-1} \sigma_j^{-1} \sigma_{j+1}^{-1} \cdots \sigma_i^{-1}.$ Compare twisted torus links in \cite{BK} with the closure of: $$w_{1,n-1} = \sigma_ 1 \sigma_2 \cdots \sigma_{n-1} \sigma_{n-1} \sigma_{n-2} \cdots \sigma_1 = (\sigma_ 1 \sigma_2 \cdots \sigma_{n-1})^n(\sigma_ 1 \sigma_2 \cdots \sigma_{n-2})^{-(n-1)} \cdots \sigma_1^{2(-1)^n} \in \mathcal{B}_n.$$

\section{The PS braid conjecture}

\hspace*{25pt}In 2012, the following conjecture addressing the potential connection between the torsion subgroups that appear in the KH of a closed braid and its braid index was posed by Przytycki and Sazdanovi\'{c} \cite{PS}.

\begin{conjecture}\
	
	\begin{enumerate}
		\item{KH of a closed $3$-braid can have only $\mathbb{Z}_2$ torsion.}
		\item{KH of a closed $4$-braid cannot have odd torsion.}
		\item{KH of a closed $4$-braid can have only $\mathbb{Z}_2$ and $\mathbb{Z}_4$ torsion.}
		\item{KH of a closed $n$-braid cannot have $\mathbb{Z}_p$ torsion for $p>n,$ where $p$ is prime.}
		\item{KH of a closed $n$-braid cannot have $\mathbb{Z}_{p^r}$ torsion for $p^r > n.$}
	\end{enumerate}

\end{conjecture}

\begin{figure}[ht]
	
	\centering
	\includegraphics[width = \textwidth]{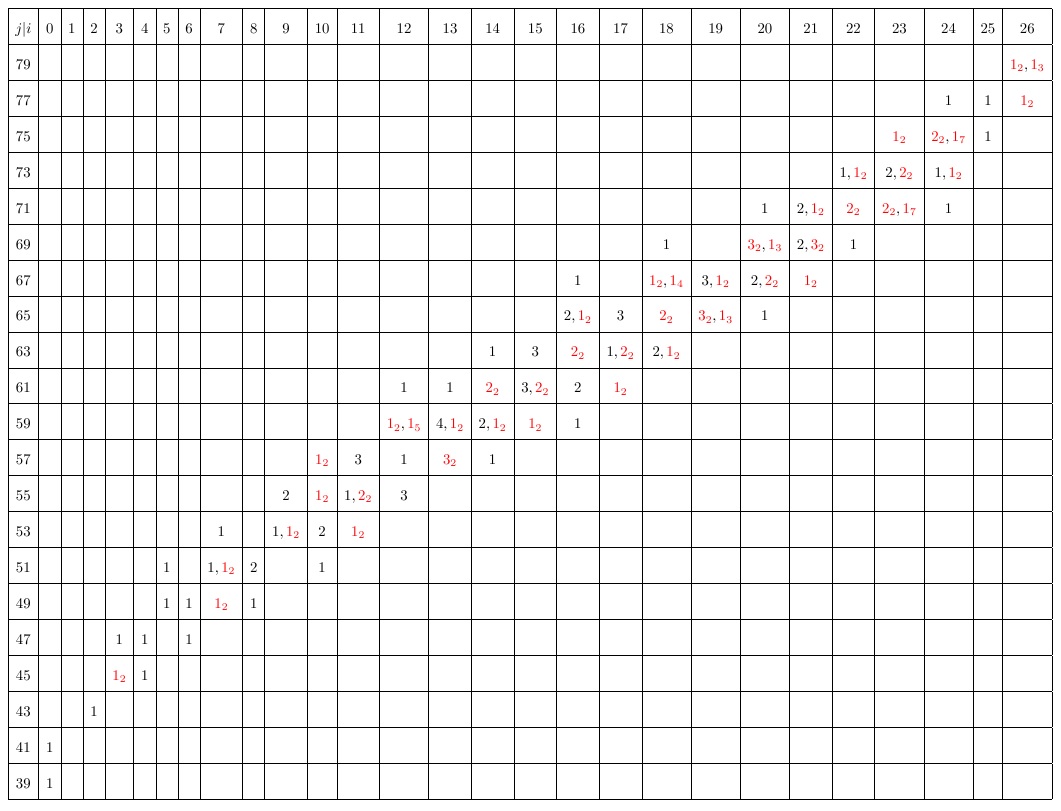}
	\caption{KH of the closure of $(\sigma_1\sigma_2\sigma_3\sigma_4\sigma_5)^7 \cdot w_{1,5}.$}
	\label{KH_with_7}
	
\end{figure}

\hspace*{25pt}It is believed that the first part of the PS braid conjecture is true. In fact, it has been proven to be true for four out of the seven families of $3$-braids \cite{CLSS,Mur}. In particular, these infinite families contain the torus links $T(3,n).$ Parts two and three were disproved by the closure of the braid: $w_{1,3}^5 \in \mathcal{B}_4$ \cite{MPSWY}. The last part was disproved by the family of links with large torsion subgroups of even order in their KH. A prominent member of this family is the flat two cabling of the smallest non-alternating knot $8_{19}.$ The final part of the PS braid conjecture is resolved by the following theorem \cite{Muk}.

\begin{theorem}\

\hspace*{25pt}The part of the PS braid conjecture stating that the order of torsion subgroups, when prime, appearing in the KH of a closed braid is bounded above by its braid index, is false. 

\end{theorem}

\begin{proof}\

\hspace*{25pt}Consider the braid: $(\sigma_1\sigma_2\sigma_3\sigma_4\sigma_5)^7 \cdot w_{1,5} \in \mathcal{B}_6$. The closure of this braid is a knot and its KH contains $\mathbb{Z}_7$ torsion in bigradings $(23,71)$ and $(24,75)$ as seen in Figure \ref{KH_with_7}. 
\end{proof}

\begin{remark}\

\hspace*{25pt}There are braids with smaller crossing numbers and braid indexes whose closures are counterexamples to the last part of the PS braid conjecture. However, all such known examples are links. One of them is: $(\sigma_1\sigma_2\sigma_3\sigma_4)^5 \cdot w_{1,4}^6 \in \mathcal{B}_5.$ The closure of this braid contains $\mathbb{Z}_7$ in its KH.

\end{remark}

\section{More knots and links with odd torsion in KH}

\hspace*{25pt}After \cite{MPSWY}, the state of affairs regarding odd torsion subgroups in KH was that infinitely many examples of odd torsion subgroups upto $\mathbb{Z}_7$ were known and the next higher one, meaning $\mathbb{Z}_9,$ was expected to show up in the KH of the torus knot of type $(9,10).$ However, the KH of this knot is beyond reach even at present due to limited computational abilities. The challenge, therefore, was to find a knot or link within computational reach with $\mathbb{Z}_9$ torsion in its KH. Surprisingly, not only was it possible to find knots and links with $\mathbb{Z}_9$ torsion in their KH, but also ones with $\mathbb{Z}_{27},$ $\mathbb{Z}_{81},$ and $\mathbb{Z}_{25}$ torsion in their KH.

\begin{theorem}\
	
\begin{enumerate}
		
		\item{The closure of $(\sigma_1\sigma_2\sigma_3\sigma_4)^5 \cdot w_{1,4}^5 \in \mathcal{B}_5$ contains $\mathbb{Z}_9$ torsion in its KH. Also, the connected sum of the torus knot of type $(5,6)$ with itself contains $\mathbb{Z}_9$ torsion in its KH.}
		\item{The closure of the braid: $(\sigma_1\sigma_2\sigma_3\sigma_4)^6(\sigma_4\sigma_5\sigma_6\sigma_7)^6(\sigma_7\sigma_8\sigma_9\sigma_{10})^6,$ the {\it overlapping} connected sum of the torus knot of type $(5,6)$ with itself twice, contains $\mathbb{Z}_{27}$ torsion in its KH.}
		\item{The connected sum of the closure of the braid: $(\sigma_1\sigma_2\sigma_3\sigma_4\sigma_5)^6\sigma_1\sigma_2\sigma_3\sigma_4$ with itself contains $\mathbb{Z}_{25}$ torsion in its KH.}
		
	\end{enumerate}

\end{theorem}

\hspace*{25pt} The article \cite{MPSWY} proposed a family of links with torsion subgroups of order $2^s$ in their KH. The following conjecture, verified upto $m = 4,$ proposes a similar family for torsion subgroups of order $3^m.$

\begin{conjecture}\

\hspace*{25 pt}The KH of the link $T(2,3) \ \#_{m} \ (\sigma_1\sigma_2\sigma_3)^4\sigma_1\sigma_2,$ for $m \in \mathbb{Z}^{+}$, contains the torsion subgroups $\mathbb{Z}_3, \mathbb{Z}_9, \ldots , \mathbb{Z}_{3^m},$ where $\#$ denotes the connected sum operation.

\end{conjecture}

\hspace*{25 pt}T(2,3), in the above conjecture, can be replaced with a copy of $(\sigma_1\sigma_2\sigma_3)^4\sigma_1\sigma_2$. However, the trefoil knot is chosen as it has a lower crossing number. In general, the computational complexity of KH is exponential. However, from some experience it is clear that for the KH of certain families of links, the computational complexity seems to be smaller. It may be useful to study this under a general framework when looking to predict the presence of larger torsion subgroups in KH in a consistent way.  Moreover, it is still unknown how non-$\mathbb{Z}_2$ torsion subgroups appear in KH. The note is concluded with the following question in relation to \cite{Wat}.

\begin{question}\
	
\hspace*{25pt}Are there families of knots and links having the same free part but different torsion subgroups in their KH? 
	
\end{question}

\section*{Acknowledgements}

\hspace*{25 pt}The computational data in this note was obtained using {\bf JavaKh}-v2 written by Scott Morrison. It is an update of Jeremy Green's {\bf JavaKh}-v1 written under the supervision of Dror Bar-Natan. The author would like to thank Adam Lowrance, J\'{o}zef H. Przytycki, Radmila Sazdanovi\'{c}, and Alexander N. Shumakovitch for their useful comments and suggestions on this project.

\

\hspace*{25 pt}The author is grateful to the Columbian College of Arts and Sciences at the George Washington University for the Dean's Dissertation Completion Fellowship.


\begin{thebibliography}{99999999}
	
\bibitem[AP]{AP} M. M. Asaeda, J. H. Przytycki, {\bf Khovanov homology: torsion and thickness}.
{\it Advances in topological quantum field theory,} 135-166, NATO Sci. Ser. II Math. Phys. Chem., 179, Kluwer Acad. Publ., Dordrecht, 2004.

\bibitem[BK]{BK} J. Birman, I. Kofman, {\bf A new twist on Lorenz links}.
{\it J. Topol.} 2 (2009), no. 2, 227-248.

\bibitem[BN]{BN} D. Bar-Natan, {\bf On Khovanov's categorification of the Jones polynomial}. {\it Algebraic Geom. Topology} 2 (2002) 337–370.

\bibitem[CLSS]{CLSS} A. Chandler, A. Lowrance, R. Sazdanovi\'{c}, V. Summers, {\bf Torsion in thin regions of Khovanov Homology}. Preprint (\url{https://arxiv.org/abs/1903.05760}).

\bibitem[Jon]{Jon} V. F. R. Jones, {\bf A polynomial invariant for knots via von Neumann algebras}. 
{\it Bull. Amer. Math. Soc.} (N. S.) 12 (1985), no. 1, 103-111.

\bibitem[Kau]{Kau} L. H. Kauffman, {\bf State models and the Jones polynomial}.
{\it Topology} 26 (1987), no. 3, 395-407.

\bibitem[Kho]{Kho} M. Khovanov, {\bf A categorification of the Jones polynomial}. 
{\it Duke Math. J.} 101 (2000), no. 3, 359-426.

\bibitem[KM]{KM} P. B. Kronheimer, T. S. Mrowka, {\bf Khovanov homology is an unknot-detector}. 
{\it Publ. Math. Inst. Hautes \'{E}tudes Sci.} No. 113 (2011), 97-208.

\bibitem[MPSWY]{MPSWY} S. Mukherjee, J. H. Przytycki, M. Silvero, X. Wang, S. Y. Yang, {\bf Search for Torsion in Khovanov Homology}. {\it Exp. Math}. 27 (2018), no. 4, 488-497.

\bibitem[Muk]{Muk} S. Mukherjee {\bf On Skein Modules and Homology Theories Related to Knot Theory}.
Thesis (Ph.D.)-The George Washington University. 2019. 124 pp. ISBN: 978-1392-08901-9
ProQuest LLC (\url{https://pqdtopen.proquest.com/pubnum/13810465.html}).

\bibitem[Mur]{Mur} K. Murasugi, {\bf On closed 3-braids}. American Mathematical Society, Providence, R.I., 1974. Memoirs of
the American Mathmatical Society, No. 151.

\bibitem[PS]{PS} J. H. Przytycki, R. Sazdanovi\'{c}, {\bf Torsion in Khovanov homology of semi-adequate links}.
{\it Fund. Math.} 225 (2014), no. 1, 277-304. 

\bibitem[Ras]{Ras} J. Rasmussen, {\bf Knot polynomials and knot homologies}.
Geometry and topology of manifolds, 261-280, {\it Fields Inst. Commun.}, 47, Amer. Math. Soc., Providence, RI, 2005.

\bibitem[Shu]{Shu} A. N. Shumakovitch, {\bf Torsion of Khovanov homology}.
{\it Fund. Math.} 225 (2014), no. 1, 343-364.

\bibitem[Tur]{Tur} P. Turner, {\it Five lectures on Khovanov homology}. {\it J. Knot Theory Ramifications} 26 (2017), no. 3, 1741009, 41 pp.

\bibitem[Vir]{Vir} O. Viro, {\bf Khovanov homology, its definitions and ramifications}. 
{\it Fund. Math.} 184 (2004), 317-342.

\bibitem[Wat]{Wat} L. Watson, {\bf Knots with identical Khovanov homology}.
{\it Algebr. Geom. Topol.} 7 (2007), 1389-1407.

\end{thebibliography}
\end{document}